\documentclass[11pt,leqno]{amsart}
\usepackage[dvipsnames]{xcolor}
\usepackage{amsmath}
\usepackage{amsfonts}
\usepackage{amssymb}
\usepackage{graphicx}
\usepackage{hyperref} 
\usepackage{a4wide} 
\usepackage{mathrsfs} 
\usepackage{relsize} 

\usepackage{mathtools,amssymb,amsthm, amsmath, amsfonts} 
\usepackage{bbm} 
\usepackage{stmaryrd,mathcomp,gensymb} 
\usepackage[makeroom]{cancel} 
\usepackage{xfrac} 
\usepackage{esint}

\theoremstyle{plain}
\newtheorem{theorem}{Theorem}

\newtheorem{corollary}[theorem]{Corollary}
\newtheorem{definition}[theorem]{Definition}
\newtheorem{lemma}[theorem]{Lemma}

\newtheorem{proposition}[theorem]{Proposition}

\newtheorem{theoremA}{Theorem}

\newtheorem{corollaryA}[theoremA]{Corollary}
\renewcommand*{\thecorollaryA}

\theoremstyle{definition}

\newtheorem{remark}[theorem]{Remark}

\numberwithin{equation}{section}
\numberwithin{theorem}{section}

\usepackage{enumerate}
\usepackage[shortlabels]{enumitem}

\newcommand{\vol}{\mathrm{vol}}

\newcommand{\Hess}{\operatorname{Hess}}

\newcommand{\rr}{\mathbb{R}}

\newcommand{\nn}{\mathbb{N}}

\newcommand{\hh}{\mathbb{H}}

\newcommand{\sect}{\operatorname{Sect}}
\newcommand{\ric}{\operatorname{Ric}}

\newcommand{\supp}{\operatorname{supp}}

\def\XXint#1#2#3{{\setbox0=\hbox{$#1{#2#3}{\int}$}
     \vcenter{\hbox{$#2#3$}}\kern-.5\wd0}}

\newcommand{\inj}{\mathrm{inj}}

\newcommand{\e}{\epsilon}

\newcommand{\vp}{\varphi}

\newcommand{\R}{\mathscr{R}}

\newcommand\cE{\mathcal{E}}

\newcommand\cH{\mathcal{H}}   
 
\newcommand\cI{\mathcal{I}}

\newcommand\cL{\mathcal{L}}

\renewcommand\mod[1]{\vert{#1}\vert}
\newcommand\bigmod[1]{\bigl\vert{#1}\bigr|}

\newcommand\norm[2]{{\Vert{#1}\Vert_{#2}}}

\newcommand\bignorm[2]{\left.{\bigl\Vert{#1}\bigr\Vert_{#2}}\right.}
\newcommand\bignormto[3]{\left.{\bigl\Vert{#1}\bigr\Vert_{#2}^{#3}}\right.}

\newcommand\bigopnorm[2]{\big|\!\big|\!\big| {#1} \big|\!\big|\!\big|_{#2}}

\newcommand\Bigopnormto[3]{\Bigl|\!\Bigl|\!\Bigl| {#1}\Bigr|\!\Bigr|\!\Bigr|_{#2}^{#3}}

\newcommand\be{\beta}

\newcommand\la{\lambda}   
      
\newcommand\si{\sigma}

\newcommand\funnyk{k\hbox to 0pt{\hss\phantom{g}}}

\newcommand\lp[1]{L^p(#1)}
\newcommand\laq[1]{L^q(#1)}
\newcommand\ld[1]{L^2(#1)}

\newcommand\lr[1]{L^r(#1)}

\newcommand\sob[3]{W^{{#1},{#2}}(#3)}

\newcommand\wh{\widehat}

\newcommand\whH{\widehat{\phantom{G}}\hbox to 0pt{\hss $H$}}

\newcommand\emspace{\hbox to 6pt{\hss}}
\newcommand\ds{\displaystyle}

\newcommand\Ric{\mathop{\rm Ric}}

\newcommand\Dom{\mathrm{Dom}}

\newcount\quantno
\everydisplay{\quantno=0}\everycr{\quantno=0}
\newcommand\quant{\advance\quantno by1
                      \ifnum\quantno=1\qquad\else\quad\fi\forall }

\newcommand\BR{\mathbb{R}} 
\newcommand\BH{\mathbb{H}}

\begin{document}


\title[Gradient and Calder\'on--Zygmund inequalities]{$L^{p}$ gradient estimates and Calder\'on--Zygmund inequalities under Ricci lower bounds}
\author{Ludovico Marini}
\address[L. Marini]{Dipartimento di Matematica e Applicazioni,
Università degli Studi di Milano-Bicocca, Via R. Cozzi 55, I-20125, Milano}
\email[Corresponding author]{l.marini9@campus.unimib.it}

\author {Stefano Meda}
\address[S. Meda]{Dipartimento di Matematica ed Applicazioni, Università degli Studi di Milano-Bicocca, Via R. Cozzi 55, I-20125, Milano}
\email{stefano.meda@unimib.it}

\author{Stefano Pigola}
\address[S. Pigola]{Dipartimento di Matematica e Applicazioni,
Università degli Studi di Milano-Bicocca, Via R. Cozzi 55, I-20125, Milano}
\email{stefano.pigola@unimib.it}

\author{Giona Veronelli}
\address[G. Veronelli]{Dipartimento di Matematica ed Applicazioni, Università degli Studi di Milano-Bicocca, Via R. Cozzi 55, I-20125, Milano}
\email{giona.veronelli@unimib.it}


\begin{abstract}
In this paper we investigate the validity of first and second order $L^{p}$ estimates for the solutions of the Poisson equation 
depending on the geometry of the underlying manifold.
We first present $L^{p}$ estimates of the gradient under the assumption that the Ricci tensor is lower bounded in a local 
integral sense and construct the first counterexample showing that they are false, in general, without curvature restrictions.
Next, we obtain $L^p$ estimates for the second order Riesz transform (or, equivalently, the validity of $L^{p}$ Calder\'on--Zygmund inequalities) on the whole scale $1<p<+\infty$ by assuming that the injectivity radius 
is positive and that the Ricci tensor is either pointwise lower bounded or non-negative in a global integral sense. 
When $1<p \leq 2$, analogous $L^p$ bounds on even higher order Riesz transforms are obtained provided that also the derivatives 
of Ricci are controlled up to a suitable order. In the same range of values of $p$, for manifolds with lower Ricci 
bounds and positive bottom of the spectrum, we show that the $L^{p}$ norm of the Laplacian controls the whole 
$W^{2,p}$-norm on compactly supported functions. 
\end{abstract}

\date{\today}
\maketitle

\section{Introduction}
The purpose of this paper is to prove some regularity results 
(see Section~\ref{s: Main results} for the precise statements) concerning solutions to the Poisson equation on 
Riemannian manifolds under comparatively weak assumptions on their geometry.  We also show that certain regularity results may be 
strongly influenced by the geometry at infinity of the manifold.
One recurrent theme in our investigation is to prove (at least some of) our results under the assumption that
the Ricci curvature satisfies appropriate $L^p$ lower bounds in place of the pointwise bounds that commonly appear in the literature. 

\smallskip
In order to place our research in perspective, we begin by making 
some comments that may help the reader orienting in this fascinating field of research.  

\smallskip
Given a function $f$ in $\lp{\BR^n}$, where $1<p<\infty$, and a distributional solution $u$ of the Poisson equation $\Delta u = f$, it is 
well known that $\partial_j\partial_\ell u$ belongs to $\lp{\BR^n}$ for every pair of integers $j$ and $\ell$ in $\{1,\ldots,n\}$ and
\begin{equation} \label{f: Poisson BRn}
\bignorm{\partial_j\partial_\ell u}{p}
\leq C \bignorm{f}{p},
\end{equation}
where $C$ does not depend on $f$.  
This regularity result may be reformulated as a boundedness result in $\lp{\BR^n}$ for the so called
{second order Riesz transform}, as follows.  
For $j$ and $\ell$ as above, consider the operator $\R_{j,\ell}$, defined, at least formally, by 
$$
\big(\R_{j,\ell} f\big)\wh{\phantom a} (\xi)
= \frac{\xi_j\, \xi_\ell}{\mod{\xi}^2}\, \wh f(\xi);
$$
$\R_{j,\ell}$ is a paradigmatic example of Calder\'on--Zygmund singular integral operator, and acts on~$f$ by convolution with a specific
principal value distribution, viz. the inverse Fourier transform of the function $\xi\mapsto \xi_j\, \xi_\ell/\mod{\xi}^2$.  
Such operators are known to be bounded on $\lp{\BR^n}$, $1<p<\infty$, and of weak type $(1,1)$ \cite{Ho}.  
By virtue of the very special structure of the Euclidean space, this is equivalent to saying that 
the operator $\nabla^2(-\Delta)^{-1}$, where $\nabla^2$ denotes the second covariant derivative associated to the Euclidean metric, 
extends to a bounded operator from $\lp{\BR^n}$ to $\lp{\BR^n; T_2\BR^n}$, the space of all $L^p$ sections of the
second order covariant tensors on $\BR^n$, endowed with the standard metric.  The operator $\nabla^2(-\Delta)^{-1}$ 
will henceforth be called \textit{second order Riesz transform}, and denoted by $\R^2$.  More generally, for each positive integer $k$, one
can consider the $k^{\mathrm{th}}$ order Riesz transform $\nabla^k(-\Delta)^{-k/2}$, denoted by $\R^k$, which is
bounded from $\lp{\BR^n}$ to $\lp{\BR^n; T_2\BR^n}$, $1<p<\infty$, and of weak type $(1,1)$.  

\smallskip
The Riesz potential $\Delta^{-1}$ is unbounded on $\lp{\BR^n}$, so that one cannot expect that a distributional solution of \eqref{f: Poisson BRn}
with $L^p$ datum $f$ belongs to $\lp{\BR^n}$.  A simple scaling argument shows that both the estimates   
$$
\bignorm{u}{p}
\leq C \bignorm{f}{p} 
\qquad\hbox{and}\qquad
\bignorm{\mod{\nabla u}}{p}
\leq C \bignorm{f}{p} 
$$
fail.  However, $\Delta^{-1}$ is a smoothing operator.  Indeed, if $n\geq 3$, then 
the Hardy--Littlewood--Sobolev inequality implies that $\Delta^{-1}$ maps $\lp{\BR^n}$ to $\lr{\BR^n}$,
where $1/r=1/p-2/n$.  Thus, distributional solutions $u$ of the Poisson equation \eqref{f: Poisson BRn} belong to $\lr{\BR^n}$, hence locally
to $\lp{\BR^n}$.  This, in turn, implies that $u$ is locally (but not globally) in the Sobolev space $\sob{2}{p}{\BR^n}$.   

\smallskip
Recall that $-\Delta$ generates a Markovian semigroup, so that its $L^p$ spectrum is contained in 
the closure of the right half plane.  In particular, for every $\tau>0$ the operator $\tau\cI-\Delta$ is invertible in $\lp{\BR^n}$, $1<p<\infty$,
equivalently
\begin{equation} \label{f: a priori modified Poisson BRn}
\norm{u}{p} 
\leq C \, \norm{\tau\, u-\Delta u}{p}
\end{equation}
whenever the right hand side is finite.  In other words, solutions to the modified Poisson equation $\Delta u - \tau \, u = f$,
with datum $f$ in $\lp{\BR^n}$, are in $\lp{\BR^n}$.  It is convenient to introduce the $k^{\textrm{th}}$ order
\textit{local Riesz transform} $\R_\tau^k := \nabla^k (\tau \cI-\Delta)^{-k/2}$.  Then the estimate \eqref{f: Poisson BRn}
may be reformulated by saying that $\R_\tau^2$ is bounded from $\lp{\BR^n}$ to $\lp{\BR^n; T_2\BR^n}$.   
Furthermore, observe that the $L^p$ boundedness of the first order Riesz transform $\nabla(-\Delta)^{-1/2}$, and the Moment inequality 
\cite[Proposition~6.6.4]{Haa} (which we can apply, for $-\Delta$ is a sectorial operator on $\lp{\BR^n}$),   imply the gradient estimate
\begin{equation} \label{f: gradient estimate BRn}
\bignorm{\mod{\nabla u}}{p}
	\leq C \bignorm{(-\Delta)^{1/2} u}{p}
	\leq C \bignormto{u}{p}{1/2} \bignormto{\Delta u}{p}{1/2}
	\leq C \big(\bignorm{u}{p}+ \bignorm{\Delta u}{p}\big).  
\end{equation}
This and \eqref{f: Poisson BRn} then yield the bound   
\begin{equation} \label{f: modified Poisson BRn}
\bignorm{u}{\sob{2}{p}{\BR^n}}
\leq C \big(\bignorm{u}{p}+ \bignorm{\Delta u}{p}\big).  
\end{equation}
It is natural to speculate how the scenario described above has to be modified as
we progressively move away from the familiar Euclidean space, by replacing $\BR^n$ with a
complete noncompact $n$ dimensional Riemannian manifold $M$, and the Laplace operator by the Laplace--Beltrami
operator, which we henceforth denote by $\Delta$.  
Clearly, the definitions of Riesz transform and local Riesz transform of order $k$ extend in an obvious way to this more general setting.
They will be denoted by $\R_\tau^k$ and $\R^k$, respectively.

\smallskip
Simple examples that illustrate how subtle the influence of the geometry at infinity of $M$ on the estimates discussed above can be
are the hyperbolic space $\BH^n$ and the connected sum $\BR^n\sharp \BR^n$ of two copies of $\BR^n$.  It is worth observing
that both $\BH^n$ and $\BR^n\sharp\BR^n$ have bounded geometry in the strongest possible sense.  

\smallskip
Since the bottom of the $L^2$ spectrum of $\Delta$ is strictly negative and its $L^1$ spectrum
is contained in the left half plane, $\Delta$ is invertible on $\lp{\BH^n}$, $1<p<\infty$, so that 
a distributional solution $u$ of the Poisson equation $\Delta u = f$, with $f$ in $\lp{\BH^n}$, automatically belongs to $\lp{\BH^n}$.  
Since the first order Riesz transform is bounded from $\lp{\BH^n}$ to $\lp{\BH^n; T_1\BH^n}$ \cite{Str1983,A1}, we can argue as in 
\eqref{f: gradient estimate BRn}, and conclude that  
\begin{equation} \label{f: Sobolev BHn}
\bignorm{u}{\sob{2}{p}{\BH^n}}
\leq C \bignorm{f}{p},
\end{equation}
an estimate which has no analogue in $\BR^n$.   

Coulhon and Duong \cite{CD1999} proved that the first order Riesz transform $\R^1$ is unbounded on $\lp{\BR^n\sharp \BR^n}$
for $p>n$.  In fact, they considered the case $n\geq 3$, but their argument can be adapted to the case where $n=2$.  Thus,
in particular, $\R^1$ is unbounded from $\lp{\BR^2\sharp\BR^2}$ to $\lp{\BR^2\sharp\BR^2; T_1 (\BR^2\sharp\BR^2)}$ for all $p>2$, a fact alien to $\BR^n$. For an interesting generalization to manifolds with finitely many Euclidean ends see \cite{CCH}.

\smallskip
Suppose now that $(M,g)$ is an $n$ dimensional Riemannian manifold and that $1<p<\infty$, and consider the problem
of determining (geometric) assumptions under which the analogues of \eqref{f: Poisson BRn}, \eqref{f: a priori modified Poisson BRn},
\eqref{f: gradient estimate BRn}, \eqref{f: modified Poisson BRn} and \eqref{f: Sobolev BHn} hold on $M$.  
It may be worth warning the reader that people in Harmonic Analysis and in Global Analysis quite often use different terminologies
to denote the same object:  in particular, the
former speak about the $L^p$ boundedness of local Riesz transforms,
whereas the latter prefer to refer to the $L^p$ Calder\'on--Zygmund inequalities
\begin{equation} \label{f: CZ M}
	\bignorm{\mod{\nabla^2 u}}{p}
	\leq C \, \big[\bignorm{u}{p}+ \bignorm{\Delta u}{p}\big]  
	\quant u \in C_c^\infty(M).   
\end{equation} 
An account of this latter approach can be found in the survey paper \cite{P2020}. The equivalence between the $L^p$ boundedness of the second order Riesz transform and the validity of an $L^p$ Calder\'on--Zygmund inequality will be formalised in Proposition \ref{prop:equiv}. The two formulations will be used interchangeably in the rest of the paper.  

\smallskip
First we look at \eqref{f: gradient estimate BRn}.  A special case of a celebrated result of D.~Bakry \cite{B} states that if the Ricci curvature
curvature of $M$ is bounded from below, then the first order local Riesz transform is bounded on $\lp{M}$ for every $p$ in $(1,\infty)$, 
equivalently there exists a constant $C$ such that 
\begin{equation} \label{f: Bakry}
\bignorm{\mod{\nabla u}}{p}
\leq C \, \big[\bignorm{(-\Delta)^{1/2} u}{p} + \bignorm{u}{p}\big]
\quant u \in C_c^\infty(M).   
\end{equation}
Thus, much as in \eqref{f: gradient estimate BRn}, we obtain the gradient estimate  
\begin{equation} \tag{GE$(p)$}\label{f: gradient estimate M}
\bignorm{\mod{\nabla u}}{p}
\leq C \, \big[\bignorm{u}{p}+ \bignorm{\Delta u}{p}\big]  
\quant u \in C_c^\infty(M).   
\end{equation}
In the case where $p>2$, this result was also obtained via probabilistic arguments by Cheng, Thalmaier and Thompson \cite{CTT}.
To the best of our knowledge it is not known whether the first order local Riesz transform is bounded from $\lp{M}$ to $\lp{M, TM}$, $1<p<2$, on any complete 
Riemannian manifold $M$.  However, Coulhon and Duong \cite{CD2003} proved that if $p \in (1, 2]$, 
then the $L^p$ gradient estimates \eqref{f: gradient estimate M} holds on any geodesically complete manifold.  A much simpler proof
thereof may be found in \cite[Lemma 1.6]{HMRV2021}.  
We emphasise that the the multiplicative estimate
$$
\bignorm{\mod{\nabla u}}{p}
	\leq C \bignormto{u}{p}{1/2} \bignormto{\Delta u}{p}{1/2}
\quant f \in C_c^\infty(M)
$$
fails if $p>2$ and $M = \BR^2\sharp \BR^2$ \cite[second remark after Theorem~4.1]{CD2003}, although $M$ has Ricci curvature bounded from
below, whence Bakry's estimate \eqref{f: Bakry} and the Moment inequality \cite[Proposition~6.6.4]{Haa} imply that \eqref{f: gradient estimate M} 
holds for every $p$ in $(1,\infty)$.  This result illustrates how sensitive of the geometry of the underlying manifold
these inequalities may be. 

\smallskip
It is natural to speculate whether the gradient estimates \eqref{f: gradient estimate M} hold for some $p>2$ under the sole assumption that 
$M$ is geodesically complete.  One of our main contributions (see Theorem~B in Section~\ref{s: Main results}) is to exhibit 
for each $p>2$ and each positive integer $n\geq 2$ an $n$ dimensional Riemannian manifold $M$ 
that does not support the gradient estimate \eqref{f: gradient estimate M}.  
According to what has been discussed above, the curvature of these manifolds is necessarily lower unbounded.
However, as we will explain in Remark \ref{rmk:optimality}, it is possible to construct examples where the negative part of the curvature grows as slowly as desired.  

Note that, as a consequence, both $\R^1$, and $\R_\tau^1$ for any $\tau>0$, are unbounded on $\lp{M}$.  

\medskip

We also prove that if $p_0>n$, and the Ricci curvature is bounded from below in an appropriate local $L^{p_0/2}$ integral sense
(see Definition~\ref{def: integral bounds} in Section~\ref{s: Main results}), then 
\eqref{f: gradient estimate M} holds for all $p$ in $(1,p_0)$ (see Theorem~A in Section~\ref{s: Main results}).  
Our condition is trivially satisfied if we assume standard pointwise lower bounds for the Ricci curvature, 
so that our result extends \cite{CTT} (which, as mentioned above, can also be obtained as an easy consequence of the $L^p$ boundedness of the first order 
local Riesz transform, proved in \cite{B}).  If, instead, $p_0$ is as above, $M$ has positive injectivity radius and nonnegative Ricci curvature
in a global $L^{p_0/2}$ integral sense (see Definition~\ref{def: integral bounds} in Section~\ref{s: Main results}), 
then \eqref{f: gradient estimate M} holds for all $p$ in $(1,\infty)$ (see Theorem~B in Section~\ref{s: Main results}).   

\smallskip
Our next set of results is concerned with Riesz transforms of even order.  
We prove the following:  
\begin{enumerate}
	\item{}
		if $M$ has positive injectivity radius and the Ricci curvature is (pointwise) bounded from below, then $\R_\tau^2$ is bounded from $\lp{M}$ to $\lp{M;T_2M}$ 
		for every $p$ in $(1,\infty)$ and $\tau>0$;
	\item{}
		if $M$ has positive injectivity radius and nonnegative Ricci curvature in the global $L^{p_0/2}$ sense
		for some $p_0>n$, then $\R_\tau^2$ is bounded from $\lp{M}$ to $\lp{M;T_2M}$  for every $p$ in $(1,\infty)$ and $\tau>0$;
	\item{}
		if $M$ has spectral gap and its and Ricci curvature is (pointwise) bounded from below, then $\R^2$ is bounded from $\lp{M}$ to $\lp{M;T_2M}$ 
		for every $p$ in $(1,2]$.  As a consequence of this and the Federer--Fleming inequality, the analogue of 
		\eqref{f: Sobolev BHn} holds on $M$;
	\item{}
		if $\ell\geq 1$, the Ricci tensor of $M$ and its derivatives up to the order $2\ell-2$ are 
		uniformly bounded, and $M$ has positive injectivity radius, then $\R^{2\ell}_\tau$ is bounded from $\lp{M}$ to $\lp{M;T_{2\ell}M}$  
		for every $p$ in $(1,2]$.  
\end{enumerate}

Note that (1) above was known under an additional pointwise upper bound on the Ricci curvature, thanks to work of G\"uneysu and the third author \cite{GP}. Subsequently, Baumgarth--Devyver--G\"uneysu \cite{BDG} proved that for $p<2$ one can replace the positivity of the injectivity radius with a bound on the whole Riemann tensor and its derivatives, as a consequence of some estimates on the covariant Riesz transforms.   Finally, a very recent and far reaching result due to Cao--Cheng--Thalmaier \cite{CCT} states that $\R^2_\tau$ is bounded from $L^p(M)$ to $\lp{M; T_2M}$ when $1<p \le2$ under the sole assumption of Ricci curvature bounded from below. There is no hope to extend this result to $p>2$ in full generality. Indeed, it is known \cite{MV0, HMRV2021} (see also \cite{DZ}) that, for every $p>2$, there exists a complete Riemannian manifold $(M,g)$ satisfying $\sect \ge0$ (in fact $\sect>0$ if $p>m$) on which $\R^2_\tau$ is unbounded in $L^p$ for every positive $\tau$. Apart from the case of Ricci-bounded geometry alluded to above, the only further set of assumptions ensuring the validity of \eqref{f: CZ M} when $p>2$ are given in \cite[Theorem 1.2]{CCT}. The manifolds considered therein must satisfy (Kato type) conditions on the curvature and its derivatives but, on the other hand, could have zero injectivity radius. Finally, in a different direction, let us recall that 
$\R^2_\tau$ is bounded from $L^2(M)$ to $L^2(M; T_2M)$ also on manifolds whose curvature is very negative, i.e. explodes polynomially to $-\infty$ in an asymptotic sense \cite{MV}.

Concerning (3) above, it was known under the additional assumption that $M$ has positive injectivity radius.  Indeed, $\R^2$
was known to be bounded from $\lp{M}$ to $\lp{M; T_2M}$ for $1<p<2$ \cite{MMV}.  Then the Federer--Fleming inequality and Bakry's estimate
allow to conclude. In a related direction, let us also point out that the study of the $L^p$ boundedness properties of $\R^2$ on complete manifolds whose full curvature tensor decays quadratically has been announced in \cite{Ca}. Finally, note that (4) was known under the additional assumption that $M$ has spectral gap (in which case an endpoint estimate
for $p=1$ was also provided).  

\medskip
In this paper we do not consider Riesz transforms of odd order $\geq 3$.  We believe that it is an interesting problem 
to find geometric conditions on $M$ under which either $\R_\tau^{2k+1}$ or $\R^{2k+1}$ is bounded on $L^p$, for some positive integer $k$.
A neat result by Anker \cite{A1} shows that if $M$ is a symmetric space of the noncompact type, then
the Riesz transforms of any order are bounded on $L^p$, $1<p<\infty$.

\medskip
The paper is organised as follows. In Section~\ref{s: Main results}, we give a precise statement of the main results. In Section~\ref{s: grad est}, we prove the $L^p$ gradient estimate \eqref{f: gradient estimate M} under local uniform $L^q$ Ricci bounds. The proof for large $p$ is based upon a related $L^\infty$ estimate \cite{DWZ2018} and a covering argument. The whole range $p>2$ is obtained via interpolation. In Section~\ref{section-GE-globalintegral}, the estimates \ref{f: gradient estimate M} are proved under global $L^q$ Ricci bounds, by exploiting the local expression in $W^{1,p}$-harmonic coordinates. To this end, the positivity of the injectivity radius is required. In Section~\ref{s: counterex}, we exhibit the (as far as we know) first examples in the literature of complete Riemannian manifolds which do not support \eqref{f: gradient estimate M} for large $p$. Such examples are obtained through a suitable conformal deformation of the Euclidean plane. Harmonic coordinates with a uniform $W^{1,q}$ bound are also the key to prove the $L^p$ boundedness of the second order Riesz transform in the case of lower bounded Ricci curvature and positive injectivity radius. This is the content of Section~\ref{s: CZ}.
Note that $W^{1, q}$-harmonic estimates for large enough $q$ imply a $C^{0, \alpha}$ control on the metric coefficients. This is an improvement on previously known bounds of the second order Riesz transform \cite{GP}, which relied on the existence of uniform $C^{1, \alpha}$-harmonic coordinates, and thus required stronger geometric assumptions. In Section~\ref{s: strong W2p}, we show that if $M$ has spectral gap and lower bounded Ricci curvature, then ($M$ supports an $L^p$ Poincar\'e inequality so that) the $W^{1,p}$ norm of a compactly supported function $u$ is bounded by $\|\Delta u\|_{L^p}$ when $1<p\le 2$. Combining with the second order bounds obtained in \cite{CCT}, this yields a control on the whole $W^{2,p}$-norm of $u$, analogous to \eqref{f: Sobolev BHn}. Finally, in Section~\ref{s: ho} we deal with the $L^p$ boundedness of higher even order local Riesz transforms. Namely, we use a trick which consists in considering the Cartesian product of $M$ with a hyperbolic plane. This allows to reduce the problem to previously known bounds for the (global) Riesz transforms on manifolds with a spectral gap.

\section{Assumptions and main results}
\label{s: Main results}

All over this paper, $M=(M,g)$ denotes a smooth complete non-compact $n$ dimensional Riemannian manifold without boundary and $p\in(1,\infty)$.

Throughout this paper, $C$ will denote a positive constant,
whose value may change from place to place. In each result, the constant $C$ will depend only on the geometric bounds assumed there, i.e. on $n$, $p$, the curvature bound and possibly the injectivity radius $i$ and the spectral gap, whenever these last two quantities are relevant.
Given a symmetric $2$-tensor field $T$, we have denoted by $\min T$ its lowest eigenvalue.\smallskip

In the literature, one can find two notions of integral curvature bounds, one of global nature and one of  uniform local nature.\smallskip

\begin{definition}\label{def: integral bounds}
	Suppose that $K\geq 0$, $R > 0$ and $1< p <+\infty$.  Set
	\begin{equation}\label{eq:Lpglobal}
		\varrho_{K}(x) :=  (\min \ric + (n-1)K^{2})_{-} (x)
	\end{equation}
(where $f_-$ denotes the negative part of $f$),
	\begin{align*}
		k(x, p, R, K) := R^2 \, \frac{\Vert \varrho_K \Vert_{L^{p}(B_R(x))}}{\mu(B_R(x))^{1/p}} 
		\quad\hbox{and}\quad
		k(p, R, K) := \sup_{x \in M} k(x, p, R, K).
	\end{align*}
	Say that:
	\begin{itemize}
		\item $M$ has Ricci curvature bounded from below by $-(n-1)K^{2}$ in the global $L^{p}$ sense if $\varrho_{K}\in L^{p}(M)$. 
		\item $M$ has an $\e>0$-amount of Ricci curvature below  $-(n-1)K^{2}$ in the $L^{p}$ sense at the scale $R$  if  $k(p,R,K)<\e$.
	\end{itemize}
\end{definition}

%

Our first main contribution is the following

\begin{theoremA}\label{thm:Lp gradient}
	Suppose that $n < p_0 < + \infty$. There exists a constant $\varepsilon = \varepsilon(p_0, n, K)> 0$ such that if 
	$k(p_0/2, 1, K) \leq \varepsilon$ for some $K\ge 0$,
	then the $L^{p}$ gradient estimate \eqref{f: gradient estimate M} holds on $M$ for every $1 < p \leq p_0$.
\end{theoremA}

\begin{remark}
\label{rmk:integral ricci vs ricci bounded from below}
Note that $\rho_K(x) = 0$ if and only if $\Ric(x) \ge- (n-1)K^2g_x$ where the inequality is intended in the sense of quadratic forms. 
In particular if the Ricci curvature satisfies the lower bound $\ric \geq -(n-1) K^2 g$, then $k(p, R, K) = 0$ for all $R> 0$ and $p \in (1, +\infty)$. 
Consequently, Theorem \ref{thm:Lp gradient} provides yet another alternative proof of the result by Cheng, Thalmaier and Thompson, \cite{CTT}, using only PDEs methods. 
On the other hand, the integral bounds we assume are in general weaker than the usual pointwise bounds; see Remark \ref{rmk:integral bounds} below.
\end{remark}

\begin{remark}
\label{rmk:extension to H2p}
If $(M, g)$ is a complete Riemannian manifold supporting an $L^p$ gradient estimate for some $p \in (1, +\infty)$, then \eqref{f: gradient estimate M} extends with the same constant to all functions in $H^{2,p}(M)$.
Indeed, if $u \in H^{2,p}(M) = \{f \in L^{p}(M): \Delta_{distr} f \in L^{p}(M)\}$, by a result of Milatovic, \cite[Appendix]{GPM2019}, there exists a sequence $\{u_k\} \subseteq C^\infty_c(M)$ such that $u_k \to u$ with respect to the $H^{2,p}$ norm. 
Applying \eqref{f: gradient estimate M} to $u_k$, we deduce that $\nabla u_k$ is Cauchy and thus converges in the space of $L^p$ vector fields. 
Testing $\nabla u_k$ against a smooth and compactly supported vector field and taking the limit shows in fact that $\nabla u_k$ converges in $L^p$ norm to the weak gradient $\nabla u$.
\end{remark}

We also obtain the following variant of Theorem \ref{thm:Lp gradient} in the case of global $L^{q}$ lower Ricci bounds.

\begin{theoremA}\label{th:GE-globalintegral}
	Suppose $r_{\inj}(M)>0$ and non-negative Ricci curvature in the global $L^{q/2}$ sense for some $n<q<+\infty$. 
	Then, for every $1<p<+\infty$, \eqref{f: gradient estimate M} holds on $M$.
\end{theoremA}

While several counterexamples to the validity of the $L^p$ Calder\'on--Zygmund inequalities have been found in recent years, \cite{GP, Li, Ve, MV}, in the case of $L^p$ gradient estimates the literature is lacking: see \cite[Section 9]{P2020} for an extensive account of the topic.
As mentioned in the introduction, using a sequence of conformal deformations on separated balls of the Euclidean plane, we are able to construct a complete Riemannian manifold on which the $L^p$ gradient estimate fails for every $2<p<+\infty$.

\begin{theoremA}
\label{thm:counterex intro}
Suppose that $n$ is an integer $\geq 2$.  For any $p>2$ there exists a complete $n$ dimensional Riemannian manifold $M$
where the $L^p$ gradient estimate \eqref{f: gradient estimate M} fails. 
\end{theoremA}


As we will explain in Section~\ref{s: counterex}, the examples in Theorem C  shows that the result of Cheng, Thalmaier and Thompson on $L^p$ gradient estimates under Ricci lower bounds, \cite{CTT} is, in fact, optimal with respect to pointwise bounds.\medskip

The next contributions of the paper will concern Riesz transforms of even order $2k\ge 2$. As announced in the introduction, adopting a different point of view all the next theorems can be restated in term of Calder\'on--Zygmund inequalities, as a consequence of the following Proposition, whose proof is deferred to Section~\ref{s: CZ}.
\begin{proposition}\label{prop:equiv}
		Let $1<p<\infty$, $\tau > 0$ and let $k\ge 1$ be an integer. The local Riesz transform $\R_\tau^{2k}$ of order $2k$ is bounded from
		$\lp{M}$ to $\lp{M;T_{2k}M}$ if and only if the $L^p$ Calder\'on--Zygmund inequality or order $2k$ 
		\begin{equation} 
		    \label{f: CZ higher order}
			\bignorm{\mod{\nabla^{2k} u}}{p}
			\leq C \, \big[\bignorm{u}{p}+ \bignorm{\Delta^k u}{p}\big]  
			\quant u \in \Dom_{L^p}(\Delta^k)
		\end{equation}
		holds on $M$, where 
		\[
		\Dom_{L^p}(\Delta^k) = \{u\in L^p(M)\ :\ \Delta^ku\in L^p(M)\}
		\]
		is the domain of the Laplacian in $L^p$.
		 
		Moreover, when $k=1$ the latter assertions are also equivalent to
	\begin{equation*} 
		\bignorm{\mod{\nabla^2 u}}{p}
		\leq C \, \big[\bignorm{u}{p}+ \bignorm{\Delta u}{p}\big]  
		\quant u \in C_c^\infty(M).
	\end{equation*}
\end{proposition}


First, we prove the $L^p$ boundedness of the local second order Riesz transform (resp. the validity of the $L^p$ Calder\'on--Zygmund inequality), on manifolds with positive injectivity radius and a lower bound on the Ricci curvature.
\begin{theoremA}\label{th:CZ}
	Suppose that $r_{\inj}(M)>0$, and either  $\ric \geq -(n-1)K^{2}$ for some $K \geq 0$, or 
	the Ricci curvature is nonnegative in the global $L^{q/2}$ sense for some $q > n$.  Then, for any $\tau>0$, $\R_\tau^2$ is bounded from
	$\lp{M}$ to $\lp{M;T_2M}$ for every $1<p<+\infty$.  
\end{theoremA}

In particular, as explained in \cite{Ve}, we have the validity of a new density result in Sobolev spaces.

\begin{corollaryA}\label{coro:density}
Under the assumptions of Theorem \ref{th:CZ}, $C^{\infty}_{c}(M)$ is dense in $W^{2,p}(M)$ for every $1<p<+\infty$.
\end{corollaryA}

\begin{remark}
As it happens for the Calder\'on--Zygmund inequality of Theorem \ref{th:CZ}, also the density result in Corollary \ref{coro:density} was already known when $1\le p \le 2$ in the wider class of complete manifolds with a pointwise lower Ricci bound (indeed, a controlled growth of the negative part of the Ricci curvature is allowed in this case). See \cite{HMRV2021} and references therein. 
\end{remark}

In the third main result of the paper we prove the strong $W^{2,p}$-estimate that, in particular, includes the  $L^p$ boundedness of the (global) Riesz transform $\R^2$.
Unfortunately, so far we have not been able to obtain its validity under integral Ricci bounds. On the positive side, 
we are able to remove the injectivity radius assumption in \cite{MMV} and to get an even stronger inequality.
We say that $M$ has spectral gap if the bottom $b$ of the $L^2$ spectrum of $-\Delta$ is strictly positive, i.e. $b>0$. We will prove the following

\begin{theoremA}\label{th:CZ'}
Suppose that $\ric \geq -(n-1)K^{2}$ and that $M$ has spectral gap. Then, for any fixed $1<p \leq 2$, the strong $W^{2,p}$-estimate 
\begin{equation}\tag{$W\!(2,p)$}\label{W2p}
	\bignorm{|\nabla u| }{L^p}+ \bignorm{|\nabla^2 u|}{L^p}   \leq C \| \Delta u \|_{L^{p}},\quad \forall u \in C^{\infty}_{c}(M),
\end{equation} holds on $M$ for some $C>0$.

Consequently, the whole $W^{2,p}$ norm of $u$ can be bounded in terms of the $L^p$ norm of its Laplacian.
\end{theoremA}

It is worth noting that Calder\'on--Zygmund estimates can be derived for higher order derivatives up to imposing more stringent conditions on the geometry of the underlying manifold.

As recalled above, it was proved in \cite{MMV} that the Riesz transform $\R^{2\ell}$ is bounded in $L^p(M)$ in the range $1<p \leq 2$, provided the geometry is bounded at the order $2\ell-2$ and $M$ has a spectral gap. We shall show how to remove the latter condition.

\begin{theoremA}\label{th:CZ-2k}
	Suppose that $\ell$ is a positive integer. Let $\tau>0$. Assume that $r_{\inj}(M)>0$ and that the covariant derivatives 
	of the Ricci tensor are uniformly bounded up to the order $2\ell-2$. Then $\R_\tau^{2\ell}$ is bounded from 
	$\lp{M}$ to $\lp{M;T_{2\ell}M}$ for every $p$ in $(1,2]$.
\end{theoremA}



\section{Gradient estimtes: local uniform \texorpdfstring{$L^q$}{Lq} Ricci bounds}\label{s: grad est}

This section is devoted to proving Theorem \ref{thm:Lp gradient}.
Preliminarly, we point out the following facts, which will be repeatedly used in the sequel.

\begin{remark}
\label{rmk:controll on k}
As noted in Section~2.3 in \cite{PW2001} for the case $K = 0$, smallness of $k(q, R_0, K)$ at a fixed scale $R_0$ implies a control on $k(q, R, K)$ for all scales $R > 0$. 
This is a consequence of a volume comparison result contained in \cite[Lemma 10]{BPS2020}. 
Indeed, if $q > n/2$ there exists $\varepsilon = \varepsilon(n, q, K)>0$ such that if $k(q, R_2, K) < \varepsilon$, then for every $0 < R_1 < R_2$ one has
\begin{equation*}
    k(q, R_1, K) \leq 4 \left(\frac{R_1}{R_2}\right)^2 \left(\frac{v_K(R_2)}{v_K(R_1)}\right)^{\frac{1}{q}} k(q, R_2, K),
\end{equation*}
where $v_K(R)$ is the volume of the geodesic ball of radius $R$ in the $n$ dimensional space form of constant curvature $K$. 
Since $v_K(R_1)\sim R_1^n$, $k(q, R_1, K) \to 0$ as $R_1 \to 0$, i.e., $k(q, R_1, K)$ can be made arbitrarily small.
See Corollary 13 in \cite{BPS2020}.

Note also that $k(p, r, K) \leq k(q, r, K)$ whenever $p\leq q$.
\end{remark}

Under the assumption that $k(p/2, 1, K)$ is small, we first prove a local $L^{p}$ gradient estimate, which is obtained integrating a local gradient estimate proved in \cite{DWZ2018}. In what follows, we use the notation
\[
\| u \|_{L^{p}(\Omega)}^{\ast} = \left( \fint_{\Omega} |u|^{p} \right)^{1/p} = \left( \frac{1}{\vol(\Omega)}\int_{\Omega} |u|^{p} \right)^{1/p}.
\]
\begin{lemma}
\label{lem:local L^p gradient}
Let $p > n$. There exists $\varepsilon =\varepsilon(n, p, K) > 0$, $C(n, p)> 1$ and $0 < R_0 \leq 1$ such that if $k(p/2, 1, K) \leq \varepsilon$, then 
\begin{equation}
    \label{eq:half harnack}
    \sup_{B_{R/2}(x)} |\nabla u |^2 \leq C R^{-2} \left[(\Vert u \Vert^*_{L^{2}(B_R(x))})^2 + (\Vert \Delta u \Vert^*_{L^{p}(B_R(x))})^2 \right]
\end{equation}
for all $0<R\leq R_0$, for all $x \in M$ and for all smooth functions $u$ on $B_1(x)$. 
Moreover, there exists a constant $D(n, p) > 0$ such that
\begin{equation}
    \label{eq:local Lp gradient}
    \bigl\Vert |\nabla u| \bigr\Vert^p_{L^{p}(B_{R/2}(x))} \leq D R^{-p} \left(\Vert u \Vert^p_{L^{p}(B_R(x))} + \Vert \Delta u \Vert^p_{L^{p}(B_R(x))} \right)
\end{equation}
for all $x \in M$, $0< R \leq R_0$ and all smooth functions $u$ on $B_1(x)$. 
\end{lemma}
\begin{proof}
By Theorem 1.9 in \cite{DWZ2018}, there exists a constant $\varepsilon_0(n, p) > 0$ independent of $R_0$ such that if $k(p/2, R_0, 0) \leq \varepsilon_0$, then \eqref{eq:half harnack} holds for all $0< R \leq R_0$.
By Remark \ref{rmk:controll on k} we know that if $k(p/2, 1, K) \leq \varepsilon$, then $k(p/2, R, K) \lesssim R^{2-n/2p}$ as $R\to 0$ and since $\varrho_0(x) \leq \varrho_K(x) + (n-1)|K|$, we have 
\begin{equation*}
    k(p/2, R, 0) \leq k(p/2, R, K) + (n-1)|K|R^2.
\end{equation*}
Hence, if we take $R_0$ small enough, then $k(p/2, R_0, 0) \leq \varepsilon_0$, which concludes the first part of the lemma.
The constant $R_0$ depends on $K, n, \varepsilon$ and $\varepsilon_0$.

From \eqref{eq:half harnack} we have
\begin{equation*}
    \sup_{B_{R/2}(x)} |\nabla u |^{p} \leq C^{p/2} R^{-p} 2^{p/2-1}\left[(\Vert u \Vert^*_{L^{2}(B_R(x))})^{p} + (\Vert \Delta u \Vert^*_{L^{p}( B_R(x))})^{p} \right].
\end{equation*}
By H\"older's inequality
\begin{equation*}
    \left(\fint_{B_R(x)} u^2\right)^{p/2} \leq \fint_{B_R(x)} u^{p},
\end{equation*}
whence
\begin{equation*}
    \int_{B_{R/2}(x)} |\nabla u|^{p} \leq C^{p/2} R^{-p} 2^{p/2-1} \frac{\vol(B_{R/2}(x))}{\vol(B_R(x))}\left(\int_{B_R(x)}  |u|^{p} + \int_{B_R(x)} | \Delta u|^{p} \right).
\end{equation*}
To conclude the proof of \eqref{eq:local Lp gradient} recall that, as a consequence of the volume comparison, $(M,g)$ satisfies a uniform local volume doubling property. See Lemma 10 and subsequent results in \cite{BPS2020}. The proof of the lemma is complete.
\end{proof}

We are now ready to prove the global $L^p$ gradient estimate.

\begin{proof}[Proof (of Theorem \ref{thm:Lp gradient})]
We start by noting that the local $L^{p_{0}}$ gradient estimate \eqref{eq:local Lp gradient}, $p_{0}>n$, extends to the whole manifold using a uniformly locally finite covering of $M$.  The existence of such covering is a formal consequence of the local volume doubling inequality, which, as we have recalled above, holds under local integral Ricci bounds. Thus, let $u \in C^\infty_c(M)$ and $\Omega = \supp(u)$ and let $0 < R \le R_0$ small enough such that $2R \leq 1$.
Here $R_0$ is the radius appearing in Lemma \ref{lem:local L^p gradient}.
By local volume doubling, there exist $x_1, \ldots, x_h \in M$ such that 
\begin{enumerate}[label=(\roman*)]
    \item $\Omega \subseteq \bigcup_{i = 1}^h B_{R/2}(x_i)$;
    \item every $x \in \Omega$ intersects at most $N$ balls
    $B_R(x_i)$.
\end{enumerate}
Then
\begin{align*}
    \int_M |\nabla u|^{p_{0}} &\leq \sum_{i = 1}^h \int_{B_{R/2}(x_i)} |\nabla u|^{p_{0}} \leq D R^{-p}  \sum_{i = 1}^h \left(\int_{B_R(x_i)} |u|^{p_{0}} + \int_{B_R(x_i)} |\Delta u|^{p_{0}}\right) \\
    &\leq D R^{-p_{0}} \int_{M} \sum_{i = 1}^h  1_{B_R(x_i)} \left(
    |u|^{p_{0}} + |\Delta u|^{p_{0}}\right) \leq D R^{-p_{0}}N \left( \int_{M} |u|^{p_{0}} + \int_M |\Delta u|^{p_{0}}\right),
\end{align*}
which proves the gradient estimate \eqref{f: gradient estimate M} with $p=p_{0}>n$.\smallskip

Recall that if $p \in (1, 2]$, then $L^p$ gradient estimates always holds on complete Riemannian manifolds \cite{CD2003}. 
We now interpolate between this and the result for $p>n$ obtained in the first part of the proof. 

It is well known that the heat semigroup is strongly continuous and contractive on 
$L^p(M)$ for all $p \in [1, +\infty)$ \cite[Theorem IV.8]{G2016}.  By the Hille-Yosida Theorem, $-1$ is in the resolvent set of its infinitesimal generator $-\Delta$.  Then $-\Delta+\cI$ is (surjective and) invertible in $\lp{M}$.
Therefore $(-\Delta+\cI)^{-1}$ is bounded on $\lp{M}$ and its range is contained in the domain of $\Delta$. 
Now, suppose that $2<p\leq n$.  Choose $q>n$ and $\theta$ in $(0,1)$, so that $1/p = \theta/q + (1-\theta)/2$.

\smallskip
On the one hand, by the first part of the proof, the operator $\nabla(-\Delta+\cI)^{-1}$ 
extends to a bounded operator from $\laq{M}$ to $\laq{M;T_1M}$.  On the other hand 
$$
\bignormto{\mod{\nabla(-\Delta+\cI)^{-1}f}}{\ld{M}}{2}
= \big((-\Delta+\cI)^{-1}f, \Delta(-\Delta+\cI)^{-1}f\big)_{\ld{M}}.
$$
Since both $(-\Delta+\cI)^{-1}$ and $\Delta(-\Delta+\cI)^{-1}$ extend to bounded operators on $\ld{M}$, the operator 
$\nabla(-\Delta+\cI)^{-1}$ extends to a bounded operator from $\ld{M}$ to $\ld{M;T_1M}$.  

\smallskip
By the Riesz--Thorin theorem $\nabla(-\Delta+\cI)^{-1}$ extends to a bounded linear operator from $L^p(M)$ to $\lp{M;T_1M}$.
As a consequence, the $L^p$ gradient estimate holds on $M$. 

The proof of the theorem is complete.

\end{proof}

\begin{remark}\label{rmk:integral bounds}
As alluded to in the introduction, the integral curvature bounds assumed here are weaker than the classical pointwise bounds. 
An easy example of a Riemannian manifold $(M,g)$ satisfying $\inf_M \min\Ric = -\infty$ but with $k(p,1,0)$ arbitrarily small, can be constructed as follows. 
We let $M=\rr^2$ endowed with the conformally flat metric $g=e^{2\varphi}dx^2$, where $\varphi$ is a smooth nonpositive function. 
In the following the sub/superscript $e$ denotes the objects taken with respect to the Euclidean metric. 
In particular $\vol_g(K)\le \vol_e(K)$ for any measurable set $K\subset \rr^2$, and $B^g_R(w)\supseteq B^e_R(w)$ for any $R>0$ and $w\in \rr^2$. 
Suppose now that $\supp \varphi \in \cup_{n\in \mathbb{ N}}B^e_{1/2}((4n,0))$. 
This guarantees that $B_1^g(w)\subseteq B_2^e(w)$ for any $w\in\rr^2$. 
Moreover, given $w\in \rr^2$, let $n_w$ be the unique integer (if any) such that $B_{1/2}^e((4n_w,0))$ intersects $B_1^e(w)$. 
Then 
\begin{equation}
\label{eq:vol lower}
\vol_g B_1^g(w)\ge \vol_g B_1^e(w) \ge \vol_g (B_1^e(w)\setminus B_{1/2}^e(4n_w,0))=\frac 34 \pi.
\end{equation}
Fix $a\in(2-\frac 2p,2)$, and $\phi_0\in C^\infty_c(B^e_{1/2}(0,0))$, we define $\varphi(x,y) = \sum_{n\in \mathbb{ N}} \phi_n(x,y)$, where $\phi_n(x,y)=n^{-a}\phi_0(n(x-4n,y))$ if $n \ge 1$. 
On the one hand, since $\Delta_e\phi_0$ attains positive values and since $\Delta_e\phi_n(x,y)=n^{2-a}\Delta_e\phi_0(n(x-4n,y))$, we have that $\Ric_g=-2\Delta_e\varphi$ is lower unbounded. 
On the other hand, 
\begin{align*}
\int_{B_1^g(w)}((\min\Ric)_-)^p d\mu_g = 2^p\int_{B_1^g(w)}((\Delta_e\varphi)_+)^p d\mu_g &\le 2^p \int_{B_2^e(w)}((\Delta_e\phi_{n_w})_+)^p dx^2\\ &= 2^p n_w^{2p-pa-2} \int_{B_{1}^e(0,0)}((\Delta_e\phi_{0})_+)^p dx^2 \\
&\le 2^p \int_{B_{1}^e(0,0)}((\Delta_e\phi_{0})_+)^p dx^2,
\end{align*}
which is uniformly bounded independently from $w$. 
Moreover, choosing an appropriate $\phi_0$, we can assume that the right hand side of the estimate above is arbitrarily small. 
Together with the uniform volume lower bound \eqref{eq:vol lower}, this proves that $k(p,1,0)<+\infty$ and can be made arbitrarily small. 
\end{remark}


\section{Gradient estimates: global \texorpdfstring{$L^{q}$}{Lq} Ricci bounds}\label{section-GE-globalintegral}

Preliminarily, we recall the following

\begin{definition}
Let $(M,g)$ be an $n$ dimensional Riemanian manifold and let $n<q<+\infty$. The $W^{1,q}$ harmonic radius at $x$, denoted by $r_{W^{1,q}}(x)$, is the supremum of all $R>0$ such that there exists a coordinate chart $\phi : B_{R}(x) \to \rr^{n}$ satisfying
\begin{itemize}
  \item [a)] $2^{-1} [\delta_{ij}] \leq [g_{ij}] \leq 2 [\delta_{ij}]$;\smallskip
  \item [b)] $R^{1-n/q}\| \partial_{k} g_{ij} \|_{L^{q}(B_{R}(x))} \leq 1$;\smallskip
  \item [c)] $\phi$ is a harmonic map.
 \end{itemize}
\end{definition}

The following result encloses in a single statement classical contributions by Anderson and Cheeger, \cite{AC}, and a more recent contribution by Hiroshima, \cite{Hi}.

\begin{theorem}\label{th:AC-Hi}
Given $n \in \nn$, $q>n$, $K\geq 0$ and $i>0$, there exists a constant $\bar r= \bar r(n,q,K,i)>0$ such that the following holds. Let $(M,g)$ be a complete, $n$ dimensional Riemannian manifold satisfying either of the following sets of assumptions:
\begin{enumerate}[a)]
 \item $r_{\inj} \geq i$ and $\ric \geq -(n-1)K^{2}$ or
 \item $r_{\inj} \geq i$ and $\ric$ is non-negative in the global $L^{q/2}$ sense, i.e., $(\min \ric)_{-} \in L^{q/2}(M)$.
\end{enumerate}
Then $r_{W^{1,q}}(z) \geq \bar r$ independently of $z \in M$.
\end{theorem}

We note that, by the Sobolev embedding, we have  for free a $C^{0,\alpha}$ control on the metric coefficients  within the ball $B_{\bar r /2}(z)$.

Finally, we observe the inclusions $B^{e}_{\bar r / 8} \subseteq \phi(B_{\bar r/4}(z)) \subseteq B^{e}_{\bar r /2}$, where $B^{e}\subseteq \rr^{n}$ denotes the Euclidean ball centered at the origin. Since, inside $B^{e}_{\bar r / 8}$, the Euclidean and the Riemannian measures are mutually controlled by absolute constants, in performing integrations in local coordinates, the chosen measure is irrelevant. \smallskip

\begin{remark}\label{rem:doubling}
We have already observed that complete manifolds with Ricci lower bounds, in the uniform local integral sense, enjoy the uniform local volume doubling property at any fixed scale. In the class of manifolds with positive injectivity radius, the same is true if we consider the case of global $L^{q}$ conditions. This follows from Croke isoperimetric estimate and volume comparison. In particular, at a sufficiently small scale, we have the existence of the covering with finite intersection multiplicity as in the proof of Theorem \ref{thm:Lp gradient}; see e.g. \cite[Proposition 1.5]{Hi}. Conversely, if one assumes a priori that $r_{W^{1,q}}(M):= \inf_{x\in M}r_{W^{1,q}}(x)>0$, then the double sided Euclidean control of the volume of the balls at a small scale implies the uniform volume doubling property, and hence the covering property.
\end{remark}

In view of  Remark \ref{rem:doubling} and of Theorem \ref{th:AC-Hi} we obtain that Theorem \ref{th:GE-globalintegral} is a direct consequence of the next result. Recall that, if $(x^{1},\cdots,x^{n})$ is a system of harmonic coordinates, then
\[
(\nabla u)^{j} = g^{jk}\partial_{k}u,\quad \Delta u = g^{ij}\partial^{2}_{ij} u,
\]
where $g=[g_{ij}]$ and $g^{-1}=[g^{ij}]$ are, respectively, the matrix of the metric coefficients and  its inverse.

\begin{theorem}\label{th:GE-harmradius}
Suppose that $r_{W^{1,q}}(M) = \bar r>0$ for some $q > n$. Then for every $1<p<+\infty$, the $L^{p}$ gradient estimate \eqref{f: gradient estimate M} holds on $M$.
\end{theorem}

\begin{proof}
 Fix $0<r<\bar r/16$. Since the metric coefficients in $W^{1,q}$-harmonic coordinates are uniformly $C^{0,\alpha}$-controlled, there exist absolute constants $C>1$ such that, for any $u \in C^{\infty}_{c}(M)$ and $0<R \leq r$,
 \[
C^{-1} \bigl\| |\nabla^{e} u| \bigr\|_{L^{p}(B^{e}_{R})} \leq  \bigl\| |\nabla u| \bigr\|_{L^{p}(B_{2R}(x))} \leq C \bigl\| |\nabla^{e}u| \bigr\|_{L^{p}(B^{e}_{4R})}
 \]
 and
 \[
 \| g^{ij}\partial^2_{ij}u \|_{L^{p}(B^{e}_{R})}  \leq C \| \Delta u \|_{L^{p}(B_{2R}(x))}.
 \]
On the other hand, by the Euclidean estimates of the gradient, \cite[Theorem 9.11]{GT}, there exists an absolute constant $C=C(n,p,R)>0$ such that
 \[
 C^{-1}\bigl\| |\nabla^{e}u| \bigr\|_{L^{p}(B^{e}_{2r})} \leq  \| u \|_{L^{p}(B^{e}_{4r})} +\| g^{ij}\partial^2_{ij}u \|_{L^{p}(B^{e}_{4r})}.
 \]
 Hence, 
\begin{align*}
 \bigl\| |\nabla u| \bigr\|_{L^{p}(B_{r}(x))} &\leq C \bigl\| |\nabla^{e}u| \bigr\|_{L^{p}(B^{e}_{2r})} \\
 &\leq  C \left(\| u \|_{L^{p}(B^{e}_{4r})} +\| g^{ij}\partial^2_{ij}u \|_{L^{p}(B^{e}_{4r})} \right)\\
 &\leq C\left( \| u\|_{L^{p}(B_{8r}(x))} +  \| \Delta u \|_{L^{p}(B_{8r}(x))}\right).
\end{align*}
Since, thanks to the uniform local doubling condition, $M$ has a countable covering by balls $\{B_{r}(x_{j})\}$ such that $\{ B_{8r}(x_{j})\}$ has finite intersection multiplicity, the global $L^{p}$ estimate follows by adding the local inequalities.
\end{proof}


\section{Counterexamples to \texorpdfstring{$L^{p}$}{Lp} gradient estimates}\label{s: counterex}

In this section, we prove Theorem~\ref{thm:counterex intro}.

\begin{proof}[Proof (of Theorem~\ref{thm:counterex intro})]
	First we prove the result in the case where $n=2$.
	Take $(\Sigma, g) = (\rr^2, \lambda^2 dx^2)$ where $dx^2$ is the usual Euclidean metric on 
	$\rr^2$ and $\lambda \in C^\infty(\Sigma)$ such that $0 < \lambda \leq 1$. 
	As above, we denote by $\Delta$ and $\nabla$ the Laplace--Beltrami operator and gradient with respect to 
	the metric $g$ while we use $\Delta_e$ and $\nabla^e$ to denote the corresponding Euclidean differential operators. 
	The spaces $L^p(\Sigma)$ are defined in terms of the Riemannian volume form $d\mu_g$,
	whereas $L^p(\rr^2)$ are the spaces with respect to the Lebesgue measure $dx^2$.  
	
	For each nonnegative integer $m$, consider the point $x_m$ in $\BR^2$, with coordinates $(m,0)$.
	Take $\lambda(x) = 1$ for all $x \in \Sigma \setminus \bigcup_{m \in \nn} B_{1/8}(x_m)$.
	Since $(\Sigma, g)$ is isometric to $(\rr^2, dx^2)$ outside of a countable union of bounded sets whose pairwise distance is uniformly lower bounded, 
	it is a complete Riemannian manifold.  Next, take $\varphi_0 \in C^\infty_c(\Sigma)$ such that
	\begin{equation*}
		\begin{cases}
			\varphi_0(u,v) = u+1 \text{ on } B_{1/4}(x_0) \\
			\supp (\varphi_0) \Subset B_{1/2}(x_0)
		\end{cases}
	\end{equation*}
	and let $\varphi_m(u, v) = \varphi_0(u-m, v)$, for all positive integers $m$.  Then, for every positive integer~$k$ define
	\begin{equation*}
		u_k := \sum_{m = 0}^k 2^{-m}\, \varphi_m.  
	\end{equation*} 
	Clearly $u_k \in C^\infty_c(\Sigma)$.  Notice that 
	\begin{equation*}
		\Vert u_k \Vert_{L^p(\Sigma)}^p 
		= \sum_{m = 0}^k\, 2^{-mp} \,\int_\Sigma |\varphi_m|^p\,  \lambda^2 \, dx 
		\leq  \sum_{m = 0}^k 2^{-mp} \, \Vert \varphi_m \Vert_{L^p(\BR^2)}^p 
		= \Vert \varphi_0 \Vert_{L^p(\rr^2)}^p \sum_{m = 0}^{+\infty} 2^{-mp} 
		< + \infty.
	\end{equation*}
	Now observe that $\Delta \varphi_m = \lambda^{-2} \Delta_e \varphi_m$.  Hence 
	\begin{align*}
		\Vert \Delta u_k \Vert_{L^p(\Sigma)}^p 
		= \sum_{m = 0}^k \, 2^{-mp}\, \int_\Sigma |\Delta \varphi_m|^p \, \lambda^2 \, dx 
		= \sum_{m = 0}^k \, 2^{-mp}\, \int_\Sigma |\Delta_e \varphi_m|^p \, \, \la^{2(1-p)} \,dx. 
	\end{align*}
	Moreover, we have that $\Delta_e \vp_m (u,v) = (\Delta_e \vp_0)(u-m, v)$.  Since $\Delta_e\vp_0$ 
	vanishes on $B_{1/4}(x_0)$, the function $\Delta_e \vp_m$ vanishes on $B_{1/4}(x_m)$.  
	This and the fact that the support of $\vp_0$ is contained in $B_{1/2}(x_0)$
	yield
	\begin{align*}
		\Vert \Delta u_k \Vert_{L^p(\Sigma)}^p 
		=  \int_{B_{1/2}(x_0) \setminus B_{1/4}(x_0)}  |\Delta_e \varphi_0|^p \, \la^{2(1-p)} \, dx
		=  \int_{B_{1/2}(x_0) \setminus B_{1/4}(x_0)}  |\Delta_e \varphi_0|^p \, dx, 
	\end{align*}
	where the last equality holds, because $\la=1$ on $B_{1/2}(x_0) \setminus B_{1/4}(x_0)$.  Altogether, we obtain that 
	\begin{align*}
		\Vert \Delta u_k \Vert_{L^p(\Sigma)}^p 
		&\leq \Vert\Delta_e\varphi_0 \Vert_{L^p(\rr^2)}^p\sum_{m = 0}^{+\infty} 2^{-mp} < + \infty.
	\end{align*}
	
	Now, recall that $p>2$ is given.  Choose $\be>1/(p-2)$, and consider $\lambda_\infty(x) := |x|^{2\beta}$ 
	in $B_\delta(x_0)$ for some $0 \leq \delta \ll 1/8$.  Note that $\mod{\nabla^e \vp_0} = 1$ on $B_{1/8}(x_0)$, whence 
	\begin{equation*} 
		\int_{B_\delta(x_0)} |\nabla^e \varphi_0|_e^p \, \lambda_\infty^{2-p} dx = 2\pi \int_0^\delta r^{1-2\beta(p-2)} dr = +\infty.
	\end{equation*}
	Here $|x| = r$ denotes the Euclidean distance from the origin. 
	Then for any $m \in \mathbb N$ we can find $\varepsilon_m > 0$, such that $\varepsilon_m \to 0$ as $m \to +\infty$, and 
	\begin{equation*}
		\int_{B_\delta(x_0)} |\nabla^e \varphi_0|^p \, (|x|^2 + \varepsilon_m)^{(2-p)\beta} dx \geq  2^{mp}.
	\end{equation*}
	For $x \in B_{1/8}(x_0)$ and $\varepsilon \in [0, 1]$ we define the function $\lambda_\varepsilon \in C^\infty(B_{1/8}(x_0))$ by 
	\begin{equation*}
		\begin{cases}
			0 < \lambda_\varepsilon \le 1 \\
			\lambda_\varepsilon(x) = (|x|^{2} + \varepsilon)^{\beta} \text{ if } x \in B_{\delta}(x_0) \\
			\supp(1 - \lambda_\varepsilon) \subseteq B_{1/8}(x_0).
		\end{cases}
	\end{equation*}
	Now define $\lambda \in C^\infty(\Sigma)$ by 
	\begin{equation*}
		\begin{cases}
			0 < \lambda \leq 1 \\
			\lambda(x) = 1 \text{ if } x \in \Sigma \setminus \bigcup_{m \in \nn} B_{1/8}(x_m) \\
			\lambda(x) = \lambda_{\varepsilon_m}(x-x_m) \text{ if } x \in B_{\delta}(x_m).
		\end{cases}
	\end{equation*}
	Then, arguing much as above, 
	\begin{align*}
		\bigl\Vert |\nabla u_k| \bigr\Vert_{L^p(\Sigma)}^p &= \int_\Sigma  \sum_{m = 0}^k \frac{|\nabla \varphi_m|^p}{2^{mp}} \lambda^2 dx 
		\geq \sum_{m = 0}^k 2^{-mp} \int_{B_\delta(x_0)} |\nabla \varphi_o|^p \lambda_m^2 dx \\
		&= \sum_{m = 0}^k 2^{-mp} \int_{B_\delta(x_0)} |\nabla^e \varphi_0|^p (|x|^2 + \varepsilon_m)^{(2-p)\beta} dx \geq k. 
	\end{align*}
	Since $\{\norm{u_k}{\lp{\Sigma}}\}$ and $\{\norm{\Delta u_k}{\lp{\Sigma}}\}$ are bounded, the gradient estimate fails on $\Sigma$.  
	
	This concludes the proof of Theorem~C in the case where $n=2$.
	
	\smallskip
	Suppose now that $n\geq 3$.  We proceed as in \cite{HMRV2021}. 
	Let $(\Sigma, g)$ be the Riemannian manifold considered above and $(N, h)$ any $n-2$ dimensional closed Riemannian manifold. 
	Consider the product manifold $M = \Sigma \times N$ and define
	$$
	v_k(x, y) = u_k(x)
	\quant  (x,y) \in \Sigma \times N.  
	$$ 
	Clearly $\lbrace v_k \rbrace \subseteq C^\infty_c(M)$.  It is straightforward to check that the sequences 
	$\{\norm{v_k}{\lp{M}}\}$ and $\{\norm{\Delta v_k}{\lp{M}}\}$ are bounded, whereas $\{\bignorm{\mod{\nabla v_k}}{\lp{M}}\}$ is unbounded. 
	Hence the gradient estimate fails on $M$.  
	
	This concludes the proof of Theorem \ref{thm:counterex intro}.
\end{proof}

\begin{remark}\label{rmk:optimality}
We observe that the choice of the sequence $\{x_m\}$ is quite arbitrary.
In particular, let $\alpha: [0, + \infty) \to [0, + \infty)$ be an arbitrary increasing function such that $\alpha(t) \to \infty$ as $t \to +\infty$. 
If we choose $x_m$ which diverges quick enough to infinity we can make the lower bound on Ricci arbitrarily small so that 
\begin{equation*}
    \Ric(x) \geq - \alpha\big(r(x)\big). 
\end{equation*}
This shows that the result by Cheng, Thalmaier and Thompson, \cite{CTT} is, in fact optimal with respect to pointwise lower bounds, as observed after the statement of Theorem \ref{thm:counterex intro}. 
\end{remark}
%
%

We also point out the following straightforward consequence of the proof of Theorem \ref{thm:counterex intro}.

\begin{corollary}
	\label{cor: consequence on sobolev spaces}
	For any $n \ge 2$ and $p > 2$, there exists a Riemannian manifold $M$ and a function $v_\infty \in H^{2, p}(M)$ such that $v_\infty \not\in W^{1, p}(M)$. 
\end{corollary}

Indeed, for $n>p$, it is enough to define $$u_\infty = \sum_{m = 0}^{+ \infty} 2^{-m}\varphi_m;$$ then $u_\infty, \Delta u_\infty \in L^p(\Sigma)$ while $|\nabla u_\infty| \not\in L^p(\Sigma)$. In particular $u_\infty \in H^{2, p}(\Sigma)$ while $u_\infty \not \in W^{1, p}(\Sigma)$. 
The case $2< p\le n$ can be dealt with the same trick as in the proof of Theorem \ref{thm:counterex intro}.


\section{Shifted Calder\'on--Zygmund inequalities}\label{s: CZ}

We begin this section by proving the equivalence between boundedness of the local Riesz transform and Calder\'on--Zygmund inequalities stated in Proposition \ref{prop:equiv}.

\begin{proof}[Proof(of Proposition \ref{prop:equiv}).]

Since $-\Delta$ generates a contraction semigroup on $\lp{M}$, the operator $-\Delta$ is sectorial in $\lp{M}$,
and the resolvent operator $(-\Delta+\tau \cI)^{-1}$ is bounded on $L^p$, by the Hille--Yosida Theorem.  
Hence so is $(-\Delta+\tau \cI)^{-k}$.  Set $\ds\psi(\la) := (\la^k+\tau)(\la+\tau)^{-k}$.
It is not hard to prove that both $\psi$ and $1/\psi$ are in the extended Dunford class $\cE_{\theta}$ for every $\theta$ in $(\pi/2,\pi)$.
By the standard functional calculus for sectorial operators, $\psi\big((-\Delta)\big)$ and $(1/\psi)\big((-\Delta)\big)$
extend to bounded operators on $\lp{M}$.

\smallskip
Suppose first that $\R_\tau^{2k}$ is bounded on $\lp{M}$, i.e. there exists a constant $C$ such that 
$$
\bignorm{\mod{\R_\tau^{2k}f}}{\lp{M}}
\leq C \bignorm{f}{\lp{M}} 
\quant f \in \lp{M}.
$$
In particular, if $u$ is in $\Dom_{L^p}\big((-\Delta)^k\big)$, then the function $f := (-\Delta+\tau \cI)^{k} u$ is in $\lp{M}$, and
$$
\bignorm{\mod{\nabla^{2k}u}}{\lp{M}}
\leq C  \bignorm{(-\Delta+\tau \cI)^{k} u}{\lp{M}} 
\leq C \, \big[\bignorm{\Delta^{k}u}{\lp{M}} +\bignorm{u}{\lp{M}} \big],
$$
where $C$ depends on $\tau$.  
The last inequality 
is a straightforward consequence of the boundedness in $L^p$ of $(1/\psi)\big((-\Delta)\big)$.
%

\smallskip
Conversely, suppose that \eqref{f: CZ higher order} holds.  Consider $f$ in $\lp{M}$.  Since $\tau$ is in the resolvent set of $(-\Delta)$,
the operator $(-\Delta+\tau \cI)^{k}$ maps $\Dom_{L^p}(\Delta^k)$ onto $\lp{M}$.  Therefore there exists $u$ in $\Dom_{L^p}\big(\Delta^k\big)$ such that $u = (-\Delta+\tau \cI)^{-k} f$.  Consequently \eqref{f: CZ higher order}, with $u$ as above yields
\begin{equation} \label{f: CZ M II}
	\bignorm{\mod{\nabla^{2k}u}}{\lp{M}}
	\leq C\, \big[\bignorm{(-\Delta+\tau \cI)^{-k} f}{p}+ \bignorm{\Delta^k (-\Delta+\tau \cI)^{-k} f}{p}\big].  
\end{equation}
Now, both $(-\Delta+\tau \cI)^{-k}$ and $\Delta^k (-\Delta+\tau \cI)^{-k}$ are bounded operators on $\lp{M}$, as $\psi\big((-\Delta)\big)$ is.  Furthermore,
standarad properties of sectorial operators imply that there exists a constant $C$ such that 
$$
\bigopnorm{(-\Delta+\tau \cI)^{-k}}{\lp{M}} 
\leq \Bigopnormto{(-\Delta+\tau \cI)^{-1}}{\lp{M}}{k} 
\leq \frac{C}{\tau^k}
\quant \tau>0
$$
and 
$$
\bigopnorm{\Delta^k (-\Delta+\tau \cI)^{-k}}{\lp{M}} 
\leq \Bigopnormto{\Delta(-\Delta+\tau \cI)^{-1}}{\lp{M}}{k} 
\leq C
\quant \tau>0.
$$
This and \eqref{f: CZ M II} yield
$$
\bignorm{\mod{\nabla^{2k}(-\Delta+\tau \cI)^{-k} f}}{\lp{M}}
\leq C\, \big[ \tau^{-k}\,  \bignorm{f}{p}+ \bignorm{f}{p}\big]
\leq C\, \max\big(1, \tau^{-k}\big) \bignorm{f}{p}, 
$$
as required.

%

\smallskip

Suppose now that $k=2$, and that \eqref{f: CZ M} holds for all $u\in C^\infty_c(M)$. Let $f\in \Dom_{L^p}(\Delta)=H^{2,p}(M)$. 
Thanks to a density result by O. Milatovic \cite[Appendix A]{GPM2019} we can take a sequence $u_j\in C^\infty_c(M)$ converging to $f$ in $H^{2,p}(M)$ as $j\to\infty$ (see Remark \ref{rmk:extension to H2p}). Hence, \eqref{f: CZ M} and \eqref{f: gradient estimate M} (which holds due to \cite[Theorem 2]{GPM2019}) implies that $u_j$ is a Cauchy sequence in $W^{2,p}(M)$, hence it converges to some limit $u_\infty\in W^{2,p}(M)$. Finally, $u_\infty=f$ and \eqref{f: CZ M}, as $W^{2,p}(M)$ continuously embeds in $H^{2,p}(M)$.
	\end{proof}

The rest of this section is devoted to prove Theorem \ref{th:CZ}. As in Section~\ref{section-GE-globalintegral} we use local estimates in $W^{1,q}$-harmomic coordinates and then glue them together thanks to the uniform local volume doubling condition.

The crucial ingredient is the following estimate of the first order term in the local expression of the Hessian of a smooth function. Recall that, if $(x^{1},\cdots,x^{n})$ is a system of harmonic coordinates, then
\[
\nabla^2_{ij}u=\Hess(u)_{ij} = \partial^{2}_{ij} u - \Gamma^{k}_{ij} \partial_{k} u
\]
where $\Gamma_{ij}^{k}$ denote the Christoffel symbols.

\begin{lemma}\label{lemma:CZ}
Let $1<p<+\infty$. Fix $z\in M$, $q> \max(n,p)$ and let $0< r = \frac{1}{4}r_{W^{1,q}}(z)$. Finally, denote by $\Gamma_{ij}^{k}$ the Christoffel symbols with respect to the $W^{1,q}$ harmonic coordinates system $\phi(x) = (x^{1},\cdots,x^{n}) : B_{ r}(z) \to U \supseteq B^{e}_{r/2}$. Then, there exists a constant $C = C(n,p,q,r)>0$ such that, for any $u \in C^{\infty}(M)$,
 \[
C^{-1} \cdot \| \Gamma^{k}_{ij} \partial_{k} u \|_{L^{p}(B^{e}_{r/2})} \leq \bigl\| |\Hess^{e} u| 	\bigr\|_{L^{p}(B^{e}_{r/2})} +	\bigl\| |\nabla u| \bigr\|_{L^{p}(B_{r}(z))}.
 \]
\end{lemma}

\begin{proof}
We apply H\"older's inequality with conjugate exponents $t=q/(q-p)$ and $ t' =  q/p$ to get
\begin{equation}\label{lemmaCZ:1}
\| \Gamma^{k}_{ij} \partial_{k} u \|_{L^{p}(B^{e}_{r/2})} \leq  \sum_{k} \| \Gamma_{ij}^{k}  \|_{L^{q}(B_{r}(z))} \cdot
\bigl\| |\nabla^{e} u| 	\bigr\|_{L^{pq/(q-p)}(B^{e}_{r/2})},\qquad\forall\,i,j=1,\dots,n.
\end{equation}
Next, we recall that the Christoffel symbols display a $C^{1}$ dependence on the metric coefficients in the form
\[
\Gamma = \frac{1}{2} g^{-1} \cdot \partial g.
\]
Since $\| g \|_{L^{\infty}}$, $\| g^{-1} \|_{L^{\infty}}$ and $\| \partial g \|_{L^{q}}$ are bounded inside $B_{ r}(z)$ (with a bound depending only on $n, q, r$), we deduce that there exists a constant $C = C(n,q, r)>0$ such that
\begin{equation}\label{lemmaCZ:2}
\| \Gamma^{k}_{ij} \|_{L^{q}(B_{ r }(z))} \leq C.
\end{equation}
It remains to take care of gradient term in \eqref{lemmaCZ:1}. To this end, for the sake of clarity, we distinguish three cases according to the values of $p$.\smallskip

\noindent ({$\mathbf{1< p < n}$}). Since
\[
\frac{pq}{q-p} < p^{\ast}:= \frac{n  p}{n-  p},
\]
we can apply directly the Sobolev(--Kondrakov) embedding theorem and deduce that, for some constant $S=S(r,p,q,n)>0$,
\[
S^{-1} \cdot \bigl\| |\nabla^{e} u| \bigr\|_{L^{pq/(q-p)}(B^{e}_{r/2})} \leq \bigl\| |\Hess^{e} u|\bigr\|_{L^{ p}(B^{e}_{r/2})} + \bigl\| |\nabla^{e} u| \bigr\|_{L^{ p}(B^{e}_{r/2})}.
\]
On the other hand, observe that
\[
\bigl\| |\nabla^{e} u| \bigr\|_{L^{p}(B^{e}_{r/2})} \leq C \bigl\| |\nabla u| \bigr\|_{L^{p}(B_{r}(z))}
\]
for some absolute constant $C>0$, whence
\begin{equation}\label{lemmaCZ:3}
 \bigl\| |\nabla^{e} u| \bigr\|_{L^{\frac{pq}{q-p}}(B^{e}_{r/2})} \leq C \left( \bigl\| |\Hess^{e} u|\bigr\|_{L^{p}(B^{e}_{r/2})}+  \bigl\| |\nabla u| \bigr\|_{L^{p}(B_{r}(z))} \right).
\end{equation}
Inserting \eqref{lemmaCZ:2} and \eqref{lemmaCZ:3} into \eqref{lemmaCZ:1}, gives the desired inequality when $1<p<n$.\smallskip

\noindent ($\mathbf{p = n}$). Let  $1<\tilde p <n=p$ be defined by
\[
\tilde p = \frac{nq}{2q-n}.
\]
Since
\[
\frac{nq}{q-n} = \frac{n \tilde p}{n-\tilde p}=: \tilde{p}^{\ast},
\]
we can apply the Sobolev embedding theorem and  the H\"older inequality to deduce that, for some constant $S=S(r,q,n)>0$,
\begin{align*}
S^{-1} \cdot \bigl\| |\nabla^{e} u| \bigr\|_{L^{nq/(q-n)}(B^{e}_{r/2})} &\leq \bigl\| |\Hess^{e} u|\bigr\|_{L^{\tilde p}(B^{e}_{r/2})} + \bigl\| |\nabla^{e} u| \bigr\|_{L^{\tilde p}(B^{e}_{r/2})}\\
&\leq |B^{e}_{r/2}|^{(n-\tilde p)/n \tilde p} \left( \bigl\| |\Hess^{e} u|\bigr\|_{L^{ n}(B^{e}_{r/2})} + \bigl\| |\nabla^{e} u| \bigr\|_{L^{ n}(B^{e}_{r/2})}\right).
\end{align*}
The conclusion follows exactly as above. 
\smallskip

\noindent ($\mathbf{p > n}$). In this case, we can use Morrey's and H\"older's inequalities to deduce that, for some constant $S = S(r,p,q,n)>0$,
\begin{align*}
\bigl\| |\nabla^{e} u| \bigr\|_{L^{pq/(q-p)}(B^{e}_{r/2})} &\leq |B^{e}_{r/2}|^{(q-p)/pq} \cdot \bigl\| |\nabla^{e} u| \bigr\|_{L^{\infty}(B^{e}_{r/2})}\\
&\leq S |B^{e}_{r/2}|^{(q-p)/qp} \left( \bigl \| |\Hess^{e} u|\bigr\|_{L^{ p}(B^{e}_{r/2})} + \bigl\| |\nabla^{e} u |\bigr\|_{L^{ p}(B^{e}_{r/2})}\right).
\end{align*}
The proof of the lemma is complete.
\end{proof}

We are now in the position to prove the following abstract result, which combined with Proposition \ref{prop:equiv}, proves Theorem \ref{th:CZ}.

\begin{theorem}
\label{th: CZ under W1q harmonic control}
Let $1<p<+\infty$. Suppose that $r_{W^{1,q}}(M)>0$ for some $q>\max(n,p)$. Then the $L^{p}$ Calder\'on--Zygmund estimate \eqref{f: CZ M} holds on $M$.
\end{theorem}

\begin{proof}
Set $\bar r = r_{W^{1,q}}(M)/4$ and let $u \in C^{\infty}_{c}(M)$. We preliminarily observe that there exists a uniform constant $C>0$ such that, for any $z \in M$,
\[
\bigl\| |\nabla^{e} u |\bigr\|_{L^{p}(B^{e}_{\bar r})} \leq C \bigl\| |\nabla u |\bigr\|_{L^{p}(B_{2 \bar r}(z))}, \qquad \| g^{ij}\partial^2_{ij}u \|_{L^{p}(B^{e}_{\bar r})} \leq C \| \Delta u \|_{L^{p}(B_{2\bar r}(z))}.
\]
Using the Euclidean Calder\'on--Zygmund estimate \cite[Theorem 9.11]{GT} joint with Lemma \ref{lemma:CZ}, we find a constant $C=C(n,p,\bar r)>0$ such that, for any $z \in M$,
\begin{align*}
\bigl\| |\Hess(u)| \bigr\|_{L^{p}(B_{\bar r /4}(z))} &\leq \bigl\| |\Hess^{e} u| \bigr\|_{L^{p}(B^{e}_{\bar r /2})}  + \sum_{ij}\| \Gamma^{k}_{ij} \partial_{k} u \|_{L^{p}(B^{e}_{\bar r/2 })} \\
 &\leq C \left( \| g^{ij}\partial^2_{ij}u \|_{L^{p}(B^{e}_{\bar r})} + \| u \|_{L^{p}(B^{e}_{\bar r})} + \bigl\| |\nabla u |\bigr\|_{L^{p}(B_{2\bar r}(z))}\right)\\
 &\leq C\left( \| \Delta u \|_{L^{p}(B_{2\bar r}(z))} + \| u \|_{L^{p}(B_{2 \bar r}(z))} + \bigl\| |\nabla u |\bigr\|_{L^{p}(B_{2\bar r}(z))}\right).
\end{align*}
Now, according to Remark \ref{rem:doubling}, we cover $M$ by a sequence of balls $\{ B_{\bar r /4}(z_{j})\}_{j \in \nn}$ with the property that the covering $\{ B_{2\bar r }(z_{j})\}_{j \in \nn}$ has finite intersection multiplicity. Summing up the local inequalities and using monotone and dominated convergence we deduce the existence of a constant $C= C(n,p,K,i)>0$ such that
\[
C^{-1} \bigl\| |\Hess(u)| 	\bigr\|_{L^{p}} \leq \| \Delta u \|_{L^{p}} + \| u \|_{L^{p}} + \bigl\| |\nabla u| \bigr\|_{L^{p}}.
\]
To conclude we apply the $L^{p}$ gradient estimates of Theorem \ref{th:GE-harmradius}. Accordingly, there exits a constant $C=C(n,p,K)>0$ such that 
\[
C^{-1 }\bigl\| |\nabla u| \bigr\|_{L^{p}} \leq \| u \|_{L^{p}} + \| \Delta u \|_{L^{p}}
\]
and this completes the proof.
\end{proof}


\section{Strong \texorpdfstring{$W^{2,p}$}{W2p}-estimates}\label{s: strong W2p}

In this section we show how to pass from a Calder\'on--Zygmund inequality to a strong $W^{2,p}$-estimate. To this end, we need to learn how to absorb the $L^{p}$ norm of a function and its gradient using the positivity of the bottom of the spectrum of the Laplacian.

\begin{definition}
 Let $(M,g)$ be a complete Riemannian manifold with $\vol(M) = +\infty$. The Cheeger constant of $M$ is defined as
 \[
 h(M) = \inf_{\Omega \Subset M, \, \partial \Omega \in C^{\infty}} \dfrac{|\Omega|}{|\partial \Omega|}.
 \]
\end{definition}

It is well known from works by Cheeger and Buser, \cite{Ch,Bu} (see also \cite{Le,DM}) that, if $(M,g)$ is a complete $n$ dimensional manifold with $\ric \geq -(n-1)K^{2}$, then the following facts are equivalent:
\begin{enumerate}[a)]
 \item (Cheeger constant) $h(M)>0$;
 \item (Spectral gap) The bottom of the spectrum $b$ of $-\Delta$ in $L^2(M)$ is strictly positive, i.e. $b>0$;
 \item ($L^{p}$ Poincar\'e) For every $1 \leq p<+\infty$, there exists a constant $C=C(n,p,K)>0$ such that
 \[
C^{-1} \| u \|^{p}_{L^{p}} \leq \bigl\| |\nabla u| \bigr\|^{p}_{L^{p}},\, \forall u\in C^{\infty}_{c}(M).
 \] 
\end{enumerate}

The main tool of this section is the following simple

\begin{lemma}\label{lemma:Lp-estimate-u}
 Let $1<p \leq 2$. Suppose that $\ric \geq -(n-1)K^{2}$ and that $M$ has spectral gap $b>0$. Then, there exists a constant $C=C(n,p,K,b)>0$ such that, for any $u \in C^{\infty}_{c}(M)$ it holds
 \[
 C^{-1}\| u \|_{L^{p}}  \leq \| \Delta u \|_{L^{p}}.
 \]
\end{lemma}

\begin{proof}
Since $b>0$, the $L^{p}$ Poincar\'e inequality tells us that
\[
C^{-1}  \| u \|_{L^{p}} \leq \bigl\| |\nabla u| \bigr\|_{L^{p}}.
\]
Now, recall from \cite{CD2003} (see also \cite{HMRV2021} for a direct proof) that, since $1<p\leq 2$, we have the validity of the multiplicative $L^{p}$ gradient estimate
\begin{align}\label{eq:CD}
C^{-1}  \bigl\| |\nabla u |\bigr\|_{L^{p}} &\leq \| u \|_{L^{p}}^{1/2} \| \Delta u \|_{L^{p}}^{1/2}\\
&\leq \varepsilon  \| u \|_{L^{p}} + \varepsilon ^{-1} \| \Delta u \|_{L^{p}},\nonumber
\end{align}
where $0<\varepsilon  \ll 1$ is arbitrary. When inserted into the Poincar\'e inequality, this latter gives
\[
C^{-1}  \| u \|_{L^{p}} \leq \| \Delta u \|_{L^{p}}
\]
where, this time, $C>0$ depends also on $\varepsilon $. This completes the proof.
\end{proof}

With this preparation we are in a position to give the

\begin{proof}[Proof (of Theorem \ref{th:CZ'})]
Since $\ric \geq -(n-1)K^{2}$ and $1<p \leq 2$, by \cite{CCT}, there exists a constant $C>0$ such that, for every $u \in C^{\infty}_{c}(M)$,
 \[
 C^{-1} \bigl\| |\Hess(u)| \bigr\|_{L^{p}} \leq \| u \|_{L^{p}} + \| \Delta u \|_{L^{p}}.
 \]
 On the other hand, the $L^{p}$ gradient estimates state that, for a suitable constant $C>0$,
 \[
 C^{-1} \bigl\| |\nabla u |	\bigr\|_{L^{p}} \leq \| u \|_{L^{p}} + \| \Delta u \|_{L^{p}}.
\]
Summarising
\[
C^{-1} \| u \|_{W^{2,p}} \leq \| u \|_{L^{p}} + \| \Delta u \|_{L^{p}}.
\]
An application of Lemma \ref{lemma:Lp-estimate-u} yields the desired strong $W^{2,p}$-estimate.
\end{proof}

\begin{remark}
Note that the assumption $p\le 2$ has been used only in the multiplicative gradient estimate \eqref{eq:CD}. While the first line of \eqref{eq:CD} is known to be false for $p>2$ (see \cite{CD2003}), we do not know if an inequality of the form $\| |\nabla u| \|_{L^{p}} \leq \varepsilon  \| u \|_{L^{p}} + C(\varepsilon ) \| \Delta u \|_{L^{p}}$ could hold when $p>2$ in the class of complete manifolds with $\ric \geq -(n-1)K^{2}$.
\end{remark}


\section{Higher order Calder\'on--Zygmund inequalities}\label{s: ho}

We start by recalling the following consequence of \cite[Theorem~5.2]{MMV}, proved by the second named author joint with Mauceri and Vallarino.

\begin{theorem} \label{t: MMV}
Suppose that $M$ has bounded geometry at the order $2\ell -2 \in \nn$, namely, 
\[
| \nabla^{j} \ric | \leq K,\, \forall j=0,\cdots,2\ell-2 \quad \text{and}\quad r_{\inj}(M)\geq i,
\]
for some constants $K\geq 0$ and $i>0$. Assume also that
$M$ has spectral gap $b>0$. Then, for any $1<p \leq 2$ there exists a constant $C=C(n,p,\ell,K,b,i)>0$ such that the global Riesz transform $\R^{2\ell}$ of order $2\ell$ is bounded from 
$\lp{M}$ to $\lp{M;T_{2\ell}M}$\end{theorem}

Actually, the result in \cite{MMV} is stronger as it establishes that the global covariant Riesz transform
$\R^{2\ell}$
is bounded as an operator from a certain Hardy space to $L^{1}$. Its $L^{p}$ boundedness for $1<p \leq 2$ then follows from an interpolation argument.


It is natural to speculate whether some of the assumptions in Theorem~\ref{t: MMV} can be removed.  Our contribution
is to allow $b$ to be zero, at the expense of considering local Riesz transforms versus the global version thereof.  This is the content of Theorem \ref{th:CZ-2k}) that we are now going to prove.

\begin{proof}[Proof (of Theorem \ref{th:CZ-2k})]
All over this proof, we denote by $\cL:=-\Delta$ the positively	 defined Laplace--Beltrami operator of the underlying manifold.
Suppose that $(M,g)$ has bounded geometry at the order $2\ell-2$. Take the standard hyperbolic plane $\hh^{2}$, and consider the Riemannian product $(M \times \hh^{2},g+g_{\hh^{2}})$. Then, denoting by $b_M=b$, $b_{\hh^2}$ and $b_{M\times\hh^2}$ the bottom of the $L^2$ spectrum of the (positive) Laplace--Beltrami operator on $M$, $\hh^2$ and $M\times \hh^2$ respectively, it holds
 \[
 b_{M \times \hh^{2}}= b_{M} + b_{\hh^{2}} \geq b_{\hh^{2}} = \frac{1}{4}.
 \]
Moreover
 \[
 | \nabla^{j} \ric_{N} | \leq \max(1,K),\, j=0,\cdots,2\ell-2,
 \]
and also
\[
r_{\inj}(M\times\hh^2) \geq r_{\inj}(M) \geq i.
\]
It follows from Theorem~\ref{t: MMV} and Proposition~\ref{prop:equiv} that, if $p$ is in $(1,2)$, there exists a constant $C>0$ such that 
	\begin{equation} \label{f: equiv I}
\bignorm{|\nabla_{M\times {\hh^2}|}^{2 \ell}w}{\lp{M\times {\hh^2}}}
\leq  C \bignorm{\cL_{M\times {\hh^2}}^{\ell}w}{\lp{M\times {\hh^2}}},
\, \forall w \in \Dom_{L^p}(\cL_{M\times {\hh^2}}).  
	\end{equation}

We apply this estimate to functions
$w$ of the form $\vp\otimes \psi$, where $\vp\in\Dom_{L^p}(\cL_M)$ and $\psi$ belongs to  $C^{\infty}_{c}({\hh^2})$.    
Since
$$
\cL_{M\times {\hh^2}}^\ell (\vp\otimes \psi)
= \sum_{j=0}^\ell \binom{\ell}{j} \,  \big(\cL_M^j \vp\big)\otimes (\cL_{\hh^2}^{\ell-j}\psi) 
$$
and 
$$
\bigmod{\nabla^{2\ell}_{M\times {\hh^2}} (\vp\otimes \psi)}_{M\times {\hh^2}}^2
= \sum_{j=0}^{2\ell} {2\binom{\ell}{j}} \bigmod{\big(\nabla^{j}_M \vp\big)\otimes\big(\nabla^{2\ell-j}_{\hh^2} \psi\big)}_{M \times {\hh^2}}^2, 
$$
by \eqref{f: equiv I} we see that 
$$
\begin{aligned}
	 \bignorm{|\nabla^{2\ell}_{M} \vp|}{\lp{M}} \bignorm{\psi}{\lp{{\hh^2}}}
	& = \bignorm{|(\nabla_M^{2\ell} \vp)\otimes \psi|}{\lp{M\times {\hh^2}}} \\
	& \leq \bignorm{|\nabla_{M\times {\hh^2}}^{2\ell} (\vp\otimes \psi)|}{\lp{M\times {\hh^2}}} 
	     \\
	& \leq C \bignorm{\cL_{M\times {\hh^2}}^\ell (\vp\otimes \psi)}{\lp{M\times {\hh^2}}} \\
	& \leq C \sum_{j=0}^\ell {\binom{\ell}{j}} \,  \bignorm{\cL_M^j \vp}{\lp{M}} \bignorm{\cL_{\hh^2}^{\ell-j}\psi}{\lp{{\hh^2}}} 
\end{aligned}
$$
Now, suppose that $\psi$ does not vanish identically on ${\hh^2}$.  Then divide both sides of the previous inequality by $\bignorm{\psi}{\lp{{\hh^2}}}$,
and obtain that 
$$
\bignorm{|\nabla_M^{2\ell} \vp|}{\lp{M}} 
\leq C \, \si_{p,\ell}\, \sum_{j=0}^\ell {\binom{\ell}{j}} \bignorm{\cL_M^j \vp}{\lp{M}} 
\quant \vp \in \lp{M}.  
$$
where
$$
\si_{p,l} 
:= \min_{0\leq j\leq l}\, \inf_{\psi\neq 0}\, \frac{\bignorm{\cL_{\hh^2}^{l-j} \psi}{\lp{{\hh^2}}}}{\bignorm{\psi}{\lp{{\hh^2}}}},
$$
is a finite constant. Now, since $\cL$ is sectorial on $\lp{M}$ (for $\cL_M$ generates the contraction semigroup $\{\cH_t\}$ on $\lp{M}$), the Moment inequality \cite[Theorem~6.6.4]{Haa} implies that 
$$
\bignorm{\cL_M^j \vp}{\lp{M}} 
\leq C \bignormto{\vp}{\lp{M}}{1-j/\ell} \bignormto{\cL_M^\ell \vp}{\lp{M}}{j/\ell}, 
$$
so that 
$$
\begin{aligned}
\sum_{j=0}^\ell \binom{\ell}{j} \bignorm{\cL_M^j \vp}{\lp{M}} 
	& \leq C \,  \big(\bignormto{\vp}{\lp{M}}{1/l} + \bignormto{\cL_M^\ell \vp}{\lp{M}}{1/\ell}\big)^\ell \\
	& \leq C \, 2^\ell  \big(\bignorm{\vp}{\lp{M}}+ \bignorm{\cL_M^\ell \vp}{\lp{M}}\big).
\end{aligned}
$$
By combining the steps above, we find that there exists a constant $C>0$ such that 
\begin{equation} \label{f: est norm intermediate}
\bignorm{|\nabla_M^{2\ell} \vp|}{\lp{M}} 
\leq C \, \big(\bignorm{\vp}{\lp{M}}+ \bignorm{\cL_M^\ell \vp}{\lp{M}}\big).
\end{equation}
A further application of Proposition \ref{prop:equiv} concludes the proof.
%
\end{proof}

\begin{remark} \label{rem: odd order}
\begin{enumerate}
\item It is natural to speculate whether the Riesz transforms of higher odd order $\R_\tau^{2\ell-1}$ are bounded on $\lp{M}$
when $\ell \geq 2$.
\item It should be possible to give an alternative proof to Theorem \ref{th:CZ-2k} using $C^{2\ell-1,\alpha}$ harmonic coordinates, which exist in our assumptions, see \cite{AC}. Such a proof would likely work also in the case $p>2$, but it would be very technical and involved, due to the large number of terms of the coordinate expression of $\nabla^{2\ell}$ to deal with; compare for instance with the analogous result for the higher order density problem in \cite{IRV}. For the sake of simplicity we decided not to investigate such an approach in this paper.
\end{enumerate}   
\end{remark}


\end{document}